\documentclass[12pt]{article}

\usepackage{graphicx}
\usepackage{amsmath}

\usepackage{mathptmx}      

\usepackage{bm}
\usepackage{amssymb}
\usepackage{amsthm}
\usepackage{tikz}
\usepackage{pgfplots}
\usepackage{verbatim}

\newtheorem{theorem}{Theorem}[section]

\makeatletter
\renewcommand{\@biblabel}[1]{#1.} 
\makeatother

\begin{document}

\title{Explicit schemes for parabolic and hyperbolic equations}

\author{Petr N. Vabishchevich\footnote{This work was supported by RFBR (project 13-01-00719)}}

\date{\small Nuclear Safety Institute, Russian Academy of Sciences, 52, B. Tulskaya, Moscow, Russia \\
      North-Eastern Federal University, 58, Belinskogo, Yakutsk, Russia \\
      E-mail: vabishchevich@gmail.com}

\maketitle

\begin{abstract}

Standard explicit schemes for parabolic equations are not very convenient for computing practice
due to the fact that they have strong restrictions on a time step.
More promising explicit schemes are associated with explicit-implicit splitting of the problem operator
(Saul'yev asymmetric schemes, explicit alternating direction (ADE) schemes, group explicit method).
These schemes belong to the class of unconditionally stable schemes, but they demonstrate bad approximation properties.
These explicit schemes are treated as schemes of the alternating triangle method and can be considered
as factorized schemes where the problem operator is splitted into the sum of
two operators that are adjoint to each other.
Here we propose a multilevel modification of the alternating triangle method, which
demonstrates better properties in terms of accuracy.
We also consider explicit schemes of the alternating triangle method for the numerical
solution of boundary value problems for hyperbolic equations of second order.
The study is based on the general theory of stability (well-posedness) for operator-difference
schemes.


\end{abstract}

\thispagestyle{empty}

\section{Introduction}

In the numerical solution of boundary value problems for evolutionary equations, emphasis is on
the approximation in time \cite{Angermann,Ascher,LeVeque}.
For parabolic equations of second order, unconditionally stable schemes are based on implicit approximations.
In this case, we must solve  the corresponding boundary value problem
for an elliptic equation at every new time level. To reduce computational costs,
explicit schemes or different variants of operator-splitting schemes are employed
\cite{Marchuk,Yanenko}. 

Explicit schemes have evident advantages over implicit schemes in terms of computational implementation.
This advantage is especially pronounced in the construction of computational algorithms
oriented to parallel computing systems. At the same time
explicit schemes have the well-known disadvantage that is associated with strong restrictions
on an admissible time step. For parabolic equations, the stability restriction
has the form $\tau < \tau _0 = O(h^2)$,
where $\tau$ is the time step and  $h$ is the step of the spatial grid 
\cite{Samarskii,SamarskiiMatusVabishchevich}.

Some promises are connected with explicit schemes, where calculations are organized in the form of traveling computations. In fact, such schemes are based on the decomposition of the problem operator into two operators, where only one of them is referred to a new time level. That is why such schemes with inhomogeneous approximation in time are called explicit-implicit schemes.
These schemes are unconditionally stable, but they have some problems with approximation.
The schemes are conditionally convergent and have an additional term $O(\tau^2 h^{-2})$
in the truncation error.

First explicit difference schemes with traveling computations for parabolic equations of second order
were proposed by Saul'yev in the book \cite{Saulev} (the book in Russian was published in 1960).
In view of explicit-implicit inhomogeneity of approximation in time, the author called 
them by asymmetric schemes.
Further fundamental result was obtained by A.A. Samarskii in the work \cite{Samarskii964},
where these schemes were treated as factorized operator-difference schemes with the additive splitting
of the problem operator (matrix) into two terms that are adjoint to each other.
Considering systems of ordinary differential equations, we split the origional matrix into the lower and upper triangular matrices, i.e., we speak of the Alternating Triangle Method (ATM).
In solving steady-state problems on the basis of such the operator splitting approach, 
we obtain iterative alternating triangle method
\cite{SamarskiiNikolaev} and the explicit alternating direction schemes \cite{Ilin}.

Further applications of explicit schemes with traveling computations for solving parabolic BVPs
can be attributed to the works performed by Evans with co-authors \cite{evans1985alternating,evans1983group}.
Taking into account peculiarities of computations, there are highlighted explicit schemes of the Group Explicit (Alternating Group Explicit) method.
Possibilities of explicit schemes under consideration for solving
BVPs for parabolic equations on parallel computers are actively discussed in the literature
(see, e.g., \cite{baolin1995age,zhuang2007alternating}).
Explicit schemes with traveling computations are also used for time-dependent convection-diffusion problems
\cite{feng2009parallel,wen2003alternating}.

In this paper, we propose a multilevel modification of the alternating triangle method (MLATM).
To improve the accuracy of  ATM schemes, we add a corrective term with the time derivative, which is taken
from the previous time level. The origional two-level scheme becomes a three-level scheme, but it
preserve stability properties (the MLATM scheme is unconditionally stable).
Because of this, the truncation error is reduced by an order of the time step magnitude:
for the second-order parabolic equation, the additional term in the truncation error is $O(\tau^3 h^{-2})$.
The stability is studied on the basis of the stability (well-posedness) theory for operator-difference schemes in finite-dimensional Hilbert spaces \cite{Samarskii,SamarskiiGulin,SamarskiiMatusVabishchevich}.

The paper is organized as follows.
In Section 2, we consider a model problem in a rectangle for a parabolic equation
of second order. Stability conditions are also formulated here for the explicit scheme.
Construction and investigation of ATM schemes is performed in Section 3.
Section 4 is the core of our work. It describes a modification of the ATM scheme
based on the transition from the two-level scheme to a three-level one.
Problems for hyperbolic equations of second order are discussed in Section 5.
In these problems, the convergence conditions of explicit schemes are acceptable if we apply
the standard version of the alternating triangular method.

\section{Model problem} 

As a typical example, we study the boundary value problem for a parabolic equation
of second order.  Let us consider a model two-dimensional parabolic problem in
a rectangle
\[
  \Omega = \{ {\bm x} \ | \ {\bm x} = (x_1, x_2),
  \quad 0 < x_{\alpha} < l_{\alpha}, \quad \alpha = 1,2 \} .
\]
An unknown function $u({\bm x},t)$ satisfies the equation
\begin{equation}\label{2.1}
   \frac{\partial u}{\partial t} 
   - \sum_{\alpha =1}^{m}
   \frac{\partial }{\partial x_\alpha} 
   \left ( k({\bm x})  \frac{\partial u}{\partial x_\alpha} \right ) = f({\bm x},t),
   \quad {\bm x}\in \Omega,
   \quad 0 < t \leq  T,
\end{equation}
where 
$\underline{k} \leq k({\bm x}) \leq  \overline{k}, \ {\bm x} \in \Omega$,
$\underline{k} > 0$.
The equation (\ref{2.1}) is supplemented with homogeneous Dirichlet boundary conditions
\begin{equation}\label{2.2}
   u({\bm x},t) = 0,
   \quad {\bm x}\in \partial \Omega,
   \quad 0 < t \leq  T.
\end{equation}
In addition, we specify the initial condition
\begin{equation}\label{2.3}
   u({\bm x},0) = u^0({\bm x}),
   \quad {\bm x}\in \Omega.
\end{equation}

In $\Omega$, we define a uniform rectangular grid:
\[
   \bar{\omega} = \{ {\bm x} \ | \ {\bm x} = (x_1, x_2),
   \quad x_\alpha = i_\alpha h_\alpha,
   \quad i_\alpha = 0,1,...,N_\alpha,
   \quad N_\alpha h_\alpha = l_\alpha \quad \alpha = 1,2 \} 
\]
and let $\omega$ be the set of interior points
($\bar{\omega} = \omega \cup \partial \omega$). 
For grid functions $y({\bm x}) = 0, \ {\bm x} \in \partial \omega$,
in the standard way, we introduce a finite-dimensional Hilbert space $H = L_2({\omega})$ 
equipped with the scalar product and norm
\[
  (y,w) \equiv \sum_{{\bm x} \in \omega}
  y({\bm x}) w({\bf x}) h_1 h_2,
  \quad \|y\| \equiv (y,y)^{1/2} .
\]
For a positive definite self-adjoint operator $D$ $(D = D^* > 0)$,
we define the space $H_D$, where
\[
 (y,w)_D \equiv (D y,w),
 \quad  \|y\|_D \equiv (y,y)_D^{1/2} . 
\] 

Let us consider a grid operator
\[
  A = D_1 + D_2.
\]
For one-dimensional grid operators
$D_{\alpha}: H \rightarrow H, \ \alpha = 1,2$,
we have
\[
  \begin{split}
  (D_1 y)({\bm x}) = & - \frac{1}{h_1} \left (
  k(x_1+0.5h_1, x_2) \frac{y(x_1+h_1, x_2)-y({\bm x})}{h_1}
  \right .  \\
  & \left . - k(x_1-0.5h_1, x_2) \frac{y({\bm x})- y(x_1-h_1, x_2)}{h_1}
  \right ), 
  \quad {\bm x} \in \omega ,
  \end{split}
\]
\[
  \begin{split}
  (D_2 y)({\bm x}) = & - \frac{1}{h_2} \left (
  k(x_1, x_2+0.5h_2) \frac{y(x_1, x_2+h_2)-y({\bm x})}{h_2}
  \right .  \\
  & \left . - k(x_1, x_2-0.5h_2) \frac{y({\bm x})- y(x_1, x_2-h_2)}{h_2}
  \right ), 
  \quad {\bm x} \in \omega .
  \end{split}
\]
In the class of sufficiently smooth coefficients $k$ and functions $u$, these operators approximate the differential operators with the second order.
In addition \cite{Samarskii,SamarskiiNikolaev}, we have in the space $H$ of grid functions:
\[
  D_{\alpha} = D^*_{\alpha},
  \quad \underline{k} \delta_{\alpha} E \leq D_{\alpha} \leq \overline{k} \Delta_{\alpha} E,
\]
\[
  \delta_{\alpha} = 
  \frac{4}{h^2_{\alpha}} \sin^2 \frac{\pi h_{\alpha}}{2 l_{\alpha}} ,
  \quad \Delta_{\alpha} = 
  \frac{4}{h^2_{\alpha}} \cos^2 \frac{\pi h_{\alpha}}{2 l_{\alpha}} ,
  \quad  \alpha = 1,2 ,
\]
where $E$ is the identity operator in $H$.
Thus
\begin{equation}\label{2.4}
 A = A^*,
 \quad \underline{k} \delta E \leq A \leq \overline{k} \Delta_{\alpha} E,
 \quad \delta = \sum_{\alpha =1}^{2} \delta_{\alpha},
 \quad \Delta = \sum_{\alpha =1}^{2} \Delta_{\alpha},
\end{equation} 

After approximation in space, using for the approximate solutions the same notation as in (\ref{2.1})--(\ref{2.3}), we obtain the Cauchy problem for the operator-differential equation
\begin{equation}\label{2.5}
  \frac{d u}{d t} + A u = f({\bm x}, t),
  \quad {\bm x} \in \omega ,
  \quad 0 < t \leq  T ,
\end{equation}
\begin{equation}\label{2.6}
  u({\bm x}, 0) = u^0({\bm x}),
  \quad {\bm x} \in \omega .
\end{equation} 

To solve numerically the problem (\ref{2.5}), (\ref{2.6}),
we start our consideration with the simplest explicit two-level scheme.
Let $\tau$ be a step of a uniform grid in time such that $y^n = y(t^n), \ t^n = n \tau$,
$n = 0,1, ..., N, \ N\tau = T$.
Let us approximate equation (\ref{2.5}) by the explicit two-level scheme
\begin{equation}\label{2.7}
  \frac{y^{n+1} - y^{n}}{\tau }
  + A y^{n} = \varphi^n,
  \quad n = 0,1, ..., N-1,
\end{equation}
where, e.g., $\varphi^n = f({\bm x}, t^{n})$.
In view of (\ref{2.6}), the operator-difference equation (\ref{2.7})
is supplemented with  the intitial condition
\begin{equation}\label{2.8}
  y^0 = u^0 .
\end{equation}
The truncation error of the difference scheme (\ref{2.7}), (\ref{2.8}) is  
$O(|h|^2 + \tau^2 + (\sigma - 0.5) \tau)$, where $|h|^2 = h_1^2 + h_2^2$. 

\begin{theorem} 
\label{t-1} 
The explicit difference scheme  (\ref{2.7}), (\ref{2.8}) is stable for
\begin{equation}\label{2.9}
 \tau \leq (1-\varepsilon) \tau_0,
 \quad  \tau_0 = \frac{2}{\|A\|} 
\end{equation} 
at any $0 < \varepsilon < 1$, 
and the finite-difference solution satisfies the estimate  
\begin{equation}\label{2.10}
  \|y^{n+1}\|^2_A \leq \|u^0\|^2_A + \frac{\tau}{2 \varepsilon}
  \sum_{k=0}^{n} \|\varphi^k\|^2 .
\end{equation}
\end{theorem}
 
\begin{proof}
Rewrite the scheme (\ref{2.7}) in the form
\[
  \left (E - \frac{\tau}{2} A \right)   \frac{y^{n+1} - y^{n}}{\tau }
  + A \frac{y^{n+1} + y^{n}}{2} = \varphi^n .
\]
Multiplying this equation scalarly in $H$ by
\[
 2 \tau y_t = 2 (y^{n+1} - y^{n}) ,
\] 
we get
\begin{equation}\label{2.11}
 2\tau \left ( \left (E - \frac{\tau}{2} A \right) y_t, y_t \right ) +
 (A y^{n+1}, y^{n+1}) - (A y^{n}, y^{n}) = 
 2\tau (\varphi^n, y_t) . 
\end{equation}
Under the restriction (\ref{2.9}) on a time step, we have
\[
  E - \frac{\tau}{2} A \geq \varepsilon E .
\]
To estimate the right-hand side of (\ref{2.11}), we use the inequality
\[
 (\varphi^n, y_t) \leq \varepsilon \|y_t\|^2 + \frac{1}{4 \varepsilon} \|\varphi^n\|^2 .
\]
From (\ref{2.11}), we arrive at the following level-wise estimate;
\[
 \|y^{n+1}\|^2_A \leq \|y^{n}\|^2_A + \frac{\tau }{2\varepsilon }  \|\varphi^n\|^2,
\]
which implies the required estimate (\ref{2.10}).
\end{proof}

Taking into account (\ref{2.4}), for the time step, we have $\tau < \tau_0$,
where, for the above-considered model problem, $\tau_0 = O(|h|^2)$. 

\section{Schemes of the alternating triangle method} 

Let us decompose the problem operator $A$ into the sum of two operators:
\begin{equation}\label{3.1}
 A = A_1 + A_2 .
\end{equation}
Individual operator terms in (\ref{3.1}) must make it possible to construct splitting schemes
based on explicit calculations.

In the alternating triangle method \cite{Samarskii964,Samarskii,SamarskiiNikolaev},
the origional matrix is splitted into the upper and lower matrices, which correspond
to the operators adjoint to each other:
\begin{equation}\label{3.2}
  A_1 = A_2^* .
\end{equation}
With regard to the problem (\ref{2.5}), (\ref{2.6}), we have
\[
  \begin{split}
  (A_1 y)({\bm x}) = & - \frac{1}{h_1} 
  k(x_1+0.5h_1, x_2) \frac{y(x_1+h_1, x_2)-y({\bm x})}{h_1}
  \\
  & - \frac{1}{h_2} 
  k(x_1, x_2+0.5h_2) \frac{y(x_1, x_2+h_2)-y({\bm x})}{h_2} ,
  \quad {\bm x} \in \omega ,
  \end{split}
\]
\[
  \begin{split}
  (A_2 y)({\bm x}) = 
  &- k(x_1-0.5h_1, x_2) \frac{y({\bm x})- y(x_1-h_1, x_2)}{h_1} , 
  \\
  & - k(x_1, x_2-0.5h_2) \frac{y({\bm x})- y(x_1, x_2-h_2)}{h_2} , 
  \quad {\bm x} \in \omega .
  \end{split}
\]
Thus, we have splitting of fluxes.

To solve the problem (\ref{2.5}), (\ref{2.6}), (\ref{3.1}), (\ref{3.2}), we can use
various splitting schemes, where the transition to a new time level is associated with solving subproblems
that are described by the individual operators $A_1$ and  $A_2$. For the above two-component splitting
(\ref{3.1}), it is natural to apply factorized additive schemes \cite{Samarskii,Vabishchevich}.
In this case, we have
\begin{equation}\label{3.3}
  (E + \sigma \tau A_1) (E + \sigma \tau A_2) \frac{y^{n+1} - y^{n}}{\tau }
  + A y^{n} = \varphi^n,
  \quad n = 0,1, ..., N-1,
\end{equation} 
where $\sigma$ is a weight parameter and $\varphi^n = f({\bm x}, \sigma t^{n+1} + (1-\sigma) t^{n})$.
The value  $\sigma = 0.5$ corresponds to the classical Peaceman-Rachford scheme 
\cite{PeacemanRachford}, whereas for $\sigma = 1$, we obtain an operator analog of the Douglas-Rachford scheme
\cite{DouglasRachford}.

\begin{theorem} 
\label{t-2} 
The factorized scheme of the alternating triangle method (\ref{3.1})--(\ref{3.3})
is unconditionally stable in $H_A$ under the restriction $\sigma \geq  0.5$ . The following a
priori estimate holds:
\begin{equation}\label{3.4}
  \|y^{n+1}\|^2_A \leq \|u^0\|^2_A + \frac{\tau}{2 }
  \sum_{k=0}^{n} \|\varphi^k\|^2 .
\end{equation}
\end{theorem} 
 
\begin{proof}
The factorized operator
\[
 B = (E + \sigma \tau A_1) (E + \sigma \tau A_2)
\] 
for the splitting  (\ref{3.1}), (\ref{3.2}) with  $\sigma \geq  0$
is self-adjoint and positive definite. More precisely, we have
\[
 B = B^* = E + \sigma \tau A +  \sigma^2 \tau^2 A_1 A_2 \geq E + \sigma \tau A .
\] 
In the above notation, the scheme (\ref{3.3}) can be written as
\begin{equation}\label{3.5}
  \left (B - \frac{\tau}{2} A \right)   \frac{y^{n+1} - y^{n}}{\tau }
  + A \frac{y^{n+1} + y^{n}}{2} = \varphi^n . 
\end{equation}
Under the restriction $\sigma \geq  0.5$, we have
\[
 B - \frac{\tau}{2} A \geq E.
\] 
Multiplication of (\ref{3.5}) scalarly in $H$ by $2 \tau y_t$ yields the equality
\[
 2\tau \left ( \left (B - \frac{\tau}{2} A \right) y_t, y_t \right ) +
 (A y^{n+1}, y^{n+1}) - (A y^{n}, y^{n}) = 
 2\tau (\varphi^n, y_t) . 
\]
Under the restriction (\ref{2.9}) on the time step, we have
\[
  E - \frac{\tau}{2} A \geq \varepsilon E .
\]
For the right-hand side, we use the inequality
\[
 (\varphi^n, y_t) \leq \|y_t\|^2 + \frac{1}{4 } \|\varphi^n\|^2 
\] 
and obtain
\[
 \|y^{n+1}\|^2_A \leq \|y^{n}\|^2_A + \frac{\tau }{2}  \|\varphi^n\|^2 ,
\] 
which immediately implies the estimate (\ref{3.4}).
\end{proof}

Special attention should be given to the investigation of the accuracy of the alternating triangle method.
The accuracy of the approximate solution of (\ref{2.5}), (\ref{2.6}) is estimated
without considering the truncation error due to approximation in space.

The convergence of the factorized scheme of the alternating triangle method
(\ref{3.1})--(\ref{3.3}) for the problem (\ref{2.5}), (\ref{2.6})
is studied in the standard way.
The equation for the error $z^n = y^n - u^n$ has the form
\[
  B \frac {z^{n+1}-z^n}{\tau} + A z^{n} = \psi^n,
  \quad n = 0,1, ..., N-1,
\]
with the truncation error $\psi^n$.
By Theorem \ref{t-2}, the error satisfies estimate
\[
  \|z^{n+1}\|^2_A \leq \frac{\tau}{2 }
  \sum_{k=0}^{n} \|\psi^k\|^2 .
\]
The truncation error has the form
\begin{equation}\label{3.6}
  \psi^n = \psi^n_{\sigma} + \psi^n_s , 
\end{equation} 
where 
\begin{equation}\label{3.7}
\begin{split}
  \psi^n_{\sigma} & = \left (\sigma - \frac{1}{2} \right ) \tau A \frac{d u}{d t}(t^{n+1/2}) + O(\tau^2) , \\
  \psi^n_s & = \sigma^2 \tau^2 A_1 A_2 \frac{d u}{d t}(t^{n+1/2}) + O(\tau^3) . 
\end{split}
\end{equation}
The first part of the truncation error $\psi^n_{\sigma}$ is standard for the conventional scheme with weights:
\[
  \frac{y^{n+1} - y^{n}}{\tau }
  + A (\sigma y^{n+1} + (1-\sigma) y^{n}) = \varphi^n,
  \quad n = 0,1, ..., N-1, 
\] 
which converges in $H_A$ with the second order with respect to $\tau$ for $\sigma =0.5$, and
only with the first order if $\sigma \neq 0.5$.

In considering the truncation error for explicit schemes of the alternating triangle method,
emphasis is on the second part $\psi^n_s$  in (\ref{3.6}), (\ref{3.7}).
Taking into account the explicit representation for the operators $A_1$ and $A_2$ in the model problem 
(\ref{2.5}), (\ref{2.6}),  we get $\psi^n_s = O(\tau^2 |h|^{-2})$.
Because of this, the operator-difference scheme (\ref{3.1})--(\ref{3.3})
for the problem (\ref{2.5}), (\ref{2.6}) has accuracy
\begin{equation}\label{3.8}
 \|z^{n+1}\|_A \leq M \left (\left (\sigma - \frac{1}{2} \right ) \tau + \tau^2 |h|^{-2} \right ) . 
\end{equation}
This conditionally convergent scheme has strong enough restrictions on a time step.
That is why it seems reasonable to modify this scheme of the alternating triangle method (\ref{3.1})--(\ref{3.3}) in order to improve accuracy by reducing error $\psi^n_s$.

\section{Multilevel alternating triangle method} 

The scheme of alternating triangle method (\ref{3.3}) is a two-level scheme.
We construct a three-level modification of this scheme, which preserves the unconditional
stability but demonstrates more acceptable estimates for accuracy.
Such schemes are called here as schemes of MLATM (Multi-Level Alternating Triangle Method).

Rewrite the scheme (\ref{3.3}) as
\[
  (E + \sigma \tau A)  \frac{y^{n+1} - y^{n}}{\tau } +
  \sigma^2\tau^2 A_1 A_2\frac{y^{n+1} - y^{n}}{\tau }
  + A y^{n} = \varphi^n .
\]
Here we have separated the term that corresponds to the standard scheme with weights from the
term proportional to $\tau^2$, which is associated with splitting.
For this, we replace the term  associated with splitting by
\[
  \sigma^2\tau^2 A_1 A_2\frac{y^{n+1} - y^{n}}{\tau } - 
  \sigma^2\tau^2 A_1 A_2\frac{y^{n} - y^{n-1}}{\tau } =
 \sigma^2\tau^3 A_1 A_2\frac{y^{n+1} - 2y^{n} + y^{n-1}}{\tau^2 }.
\]
After this modification the MLATM scheme takes the form
\begin{equation}\label{4.1}
  (E + \sigma \tau A) ) \frac{y^{n+1} - y^{n}}{\tau } +
  \sigma^2\tau^3 A_1 A_2\frac{y^{n+1} - 2y^{n} + y^{n-1} }{\tau^2 }
  + A y^{n} = \varphi^n .
\end{equation}
As in the standard ATM scheme (\ref{3.3}), the transition to a new time level in (\ref{4.1})
involves the solution of the problem
\[
  (E + \sigma \tau A_1) (E + \sigma \tau A_2) y^{n+1} = \xi^{n} .
\]
For the truncation error, now we have the representation (\ref{3.3}), where
\begin{equation}\label{4.2}
\begin{split}
  \psi^n_{\sigma} & = \left (\sigma - \frac{1}{2} \right ) \tau A \frac{d u}{d t}(t^{n+1/2}) + O(\tau^2) , \\
  \psi^n_s & = \sigma^2 \tau^3 A_1 A_2 \frac{d^2 u}{d t^2}(t^{n+1/2}) + O(\tau^4) . 
\end{split}
\end{equation}
Thus, the error associated with splitting $ \psi^n_s $ decreases by an order of $\tau$.
If we use the MLATM scheme for the splitting  (\ref{3.1}),  (\ref{3.2}) for the approximate
solution of the problem (\ref{2.1})--(\ref{2.3}) (explicit schemes), then the truncation error is $\psi^n_s = O(\tau^3 |h|^{-2})$. 

Our main result is the following.

\begin{theorem} 
\label{t-3} 
The scheme of the multilevel alternating triangle method (\ref{3.1}), (\ref{3.2}),  (\ref{4.2})
is unconditionally stable under the restriction $\sigma \geq  0.5$. The following a
priori estimate holds:
\begin{equation}\label{4.3}
  \mathcal{E}_{n+1} \leq 
  \mathcal{E}_{n} +
  \frac{\tau}{2}  \|\varphi^n \|^2_{(E + \sigma \tau A)^{-1}} ,
\end{equation}
where
\[
 \mathcal{E}_{n+1} =
 \left \| \frac{y^{n+1} + y^{n}}{2 } \right \|^2_A +
 \left \| \frac{y^{n+1} - y^{n}}{\tau } \right \|^2_{\frac{\tau}{2} E + \sigma^2\tau^3 A_1 A_2 + \frac{\tau^2}{4} (2\sigma - 1) A } .
\] 
\end{theorem} 
 
\begin{proof}
Taking into account that 
\[
  \frac{y^{n+1} - y^{n}}{\tau } = 
  \frac{y^{n+1} - y^{n-1}}{2\tau } +
  \frac{\tau}{2} \frac{y^{n+1} - 2 y^{n} + y^{n-1}}{\tau^2},
\] 
we write the scheme (\ref{4.1}) in the form 
\begin{equation}\label{4.4}
  C \frac{y^{n+1} - y^{n-1}}{2\tau } +
  G \frac{y^{n+1} - 2 y^{n} + y^{n-1}}{\tau^2} +
  A y^{n} =
  \varphi^{n} ,
\end{equation} 
where 
\[
 C = E + \sigma \tau A,
\] 
\[
 G = \frac{\tau}{2} (E + \sigma \tau A) + \sigma^2\tau^3 A_1 A_2.
\] 

By
\[
  y^{n} = 
  \frac{1}{4} (y^{n+1} + 2 y^{n} + y^{n-1}) -
  \frac{1}{4} (y^{n+1} - 2 y^{n} + y^{n-1}) ,
\] 
we can rewrite (\ref{4.4}) as 
\begin{equation}\label{4.5}
\begin{split}
  C \frac{y^{n+1} - y^{n-1}}{2\tau } & +
  \left (G - \frac{\tau^2}{4} A \right ) 
  \frac{y^{n+1} - 2 y^{n} + y^{n-1}}{\tau^2} \\
  & +
  A \frac{y^{n+1} - 2 y^{n} + y^{n-1}}{4} =
  \varphi^{n} .
\end{split}
\end{equation} 
Let  
\[
  v^{n} = \frac{1}{2} (y^{n} + y^{n-1}),
  \quad w^{n} = \frac{y^{n} - y^{n-1}}{\tau} ,
\]
then (\ref{4.5}) can be written in the form 
\begin{equation}\label{4.6}
  C \frac{w^{n+1} + w^{n}}{2} 
  + R \frac{w^{n+1} - w^{n}}{\tau } +
  \frac{1 }{2} A ( v^{n+1} + v^{n})  =
  \varphi^{n} ,
\end{equation}
where 
\[
  R = G - \frac{\tau^2}{4} A .
\] 
Multiplying scalarly both sides of (\ref{4.6}) by  
\[
  2 (v^{n+1} - v^{n}) =
  \tau (w^{n+1} + w^{n}) ,
\]
we get the equality
\begin{equation}\label{4.7}
\begin{split}
  \frac{\tau}{2} 
  ( C (w^{n+1} & + w^{n}), w^{n+1} + w^{n}) +
  ( R (w^{n+1} - w^{n}), w^{n+1} + w^{n}) 
  \\
  & +
  ( A (v^{n+1} + v^{n}), v^{n+1} - v^{n}) =
  \tau (\varphi^{n}, w^{n+1} + w^{n} ) .
\end{split}
\end{equation}
To estimate the right-hand side, we use the inequality 
\[
  ( \varphi^{n}, w^{n+1} + w^{n} ) \leq 
  \frac{1 }{2} 
  ( C (w^{n+1} + w^{n}) +
  \frac{1}{2} 
  (C^{-1} \varphi^n, \varphi^n) .
\]
This makes it possible to get from (\ref{4.7}) the inequality 
\begin{equation}\label{4.8}
  \mathcal{E}_{n+1} \leq 
  \mathcal{E}_{n} +
  \frac{\tau}{2} 
  (C^{-1} \varphi^n, \varphi^n) ,
\end{equation}
where we use the notation 
\[
  \mathcal{E}_{n} = 
  ( A v^{n}, v^{n})
  + ( R w^{n}, w^{n}) .
\]

The inequality (\ref{4.8}) is the desired a priori estimate, 
if we show that $\mathcal{E}_{n}$ defines the squared norm of the difference solution.
By the positive definiteness of $A$, it is sufficient to require a 
positiveness of the operator $R$.
In the above notation, we have 
\[
  R =  \frac{\tau}{2} (E + \sigma \tau A) + \sigma^2\tau^3 A_1 A_2 
  - \frac{\tau^2}{4} A 
  > \frac{\tau^2}{4} (2\sigma - 1) A .
\] 
Thus, $R > 0$ for $\sigma \geq 0.5$.
This concludes the proof.
\end{proof}

\section{Hyperbolic equations} 

Special attention should be given to the problem of constructing explicit schemes of the alternating triangle method  for hyperbolic equations of second order. As a model problem, we will consider the boundary value problem  in a rectangle $\Omega$ for the equation
\begin{equation}\label{5.1}
   \frac{\partial^2 u}{\partial t^2} 
   - \sum_{\alpha =1}^{m}
   \frac{\partial }{\partial x_\alpha} 
   \left ( k({\bm x})  \frac{\partial u}{\partial x_\alpha} \right ) = f({\bm x},t),
   \quad {\bm x}\in \Omega,
   \quad 0 < t \leq  T .
\end{equation}
The equation (\ref{5.1}) is supplemented with the boundary condition (\ref{2.2}) and two initial conditions:
\begin{equation}\label{5.2}
   u({\bm x},0) = u^0({\bm x}),
   \quad  \frac{\partial u}{\partial t} ({\bm x},0) = v^0({\bm x}),
   \quad {\bm x}\in \Omega.
\end{equation}

After approximation in space (see (\ref{2.5}), (\ref{2.6})),
from the problem  (\ref{2.2}),  (\ref{5.1}),  (\ref{5.2}), we arrive at the problem
\begin{equation}\label{5.3}
  \frac{d^2 u}{d t^2} + A u = f({\bm x}, t),
  \quad {\bm x} \in \omega ,
  \quad 0 < t \leq  T ,
\end{equation}
\begin{equation}\label{5.4}
  u({\bm x}, 0) = u^0({\bm x}),
  \quad \frac{d u}{d t} ({\bm x}, 0) = v^0({\bm x}),
  \quad {\bm x} \in \omega .
\end{equation} 
For the operator $A$,  the splitting (\ref{3.1}), (\ref{3.2}) takes place.

The scheme of the alternating triangle method for the problem (\ref{3.1}), (\ref{3.2}), (\ref{5.3}), (\ref{5.4}) is written \cite{Vabishchevich} like this:
\begin{equation}\label{5.5}
  G \frac{y^{n+1} - 2 y^{n} + y^{n-1}}{\tau^2} +
  A y^{n} =
  \varphi^{n} ,
  \quad n = 1,2, ..., N-1, 
\end{equation} 
where $y^0, y^1$ are prescribed. The factorized operator $G$ has the form
\begin{equation}\label{5.6}
 G = (E + \sigma \tau^2 A_1) (E + \sigma \tau^2 A_2).
\end{equation} 

For the scheme (\ref{5.5}), (\ref{5.6}), the truncation error associated with splitting is
\[
  \psi^n_s = \sigma^2 \tau^4 A_1 A_2 \frac{d^2 u}{d t^2}(t^{n}) + O(\tau^5) . 
\]
In the numerically solving problem (\ref{2.2}),  (\ref{5.1}),  (\ref{5.2}),
the explicit scheme  (\ref{5.5}), (\ref{5.6}) has the truncation error $\psi^n_s = O(\tau^4 |h|^{-2})$. 
Such the truncation error is appropriate for many applied problems. This allows us to restrict ourselves to the classical version of explicit schemes for the alternating triangle method
without the multilevel modification. It remains to obtain the condition for stability of the scheme (\ref{5.5}), (\ref{5.6}).

In the above notation, the scheme (\ref{5.5}), (\ref{5.6}) can be written as
\begin{equation}\label{5.7}
  R \frac{w^{n+1} - w^{n}}{\tau } +
  \frac{1 }{2} A ( v^{n+1} + v^{n})  =
  \varphi^{n} .
\end{equation}
In our case, we have
\begin{equation}\label{5.8}
 R = E + \left ( \sigma - \frac{1}{4} \right )\tau^2 A + \sigma^2\tau^4 A_1 A_2 \geq E 
\end{equation} 
under the restriction $\sigma \geq 0.25$.

Similarly to (\ref{4.7}), (\ref{4.8}), from (\ref{5.7}), we get
\begin{equation}\label{5.9}
  \mathcal{E}_{n+1} =
  \mathcal{E}_{n} +
  \tau (\varphi^{n}, w^{n+1} + w^{n} ) .
\end{equation}
For the right-hand side of (\ref{5.9}), we apply the estimates
\[
 \tau (\varphi^{n}, w^{n}) \leq 
 \frac{\tau }{2} \| \varphi^{n} \|^2_{R^{-1}} +
 \frac{\tau }{2} \| w^{n} \|^2_{R}  ,
\] 
\[
 \tau (\varphi^{n}, w^{n+1}) \leq 
 \frac{\tau }{2} \left (1 + \frac{\tau }{2} \right ) \| \varphi^{n} \|^2_{R^{-1}} +
 \frac{\tau }{2}  \left (1 + \frac{\tau }{2} \right )^{-1} \| w^{n+1} \|^2_{R}  .
\] 
Besides, for all $\varepsilon > 0$, we have
\[
 1 + \varepsilon \tau < \exp(\varepsilon \tau ) .
\] 
In view of these estimates, from (\ref{5.9}), it follows the level-wise estimate
\begin{equation}\label{5.10}
  \mathcal{E}_{n+1} = \exp(\tau) \mathcal{E}_{n} + 
  \exp(0.75 \tau) \tau \|\varphi^{n} \|^2_{R^{-1}} ,
\end{equation}
which ensures the stability of the solution with respect to the initial data and the right-hand side.
This proves the following statement.

\begin{theorem} 
\label{t-4} 
The scheme of the alternating triangle method (\ref{3.1}), (\ref{3.2}),  (\ref{5.5}),  (\ref{5.6})
is unconditionally stable  under the restriction $\sigma \geq  0.25$. 
The solution satisfies the estimate (\ref{5.8}),  (\ref{5.10}), where
\[
 \mathcal{E}_{n+1} =
 \left \| \frac{y^{n+1} + y^{n}}{2 } \right \|^2_A +
 \left \| \frac{y^{n+1} - y^{n}}{\tau } \right \|^2_{R} .
\] 
\end{theorem}

\end{document}